\documentclass[11pt]{article}
\usepackage[utf8]{inputenc}
\usepackage[T1]{fontenc}
\usepackage{lmodern}
\usepackage{geometry}
\usepackage{amsmath,amsthm,amssymb,mathtools}
\usepackage{microtype}
\usepackage{bm}
\usepackage{tikz-cd}
\usepackage{enumitem}
\usepackage{booktabs}
\usepackage{hyperref}
\hypersetup{colorlinks=true,linkcolor=blue,citecolor=teal,urlcolor=blue}

\newcommand{\QQ}{\mathbb{Q}}
\newcommand{\RR}{\mathbb{R}}
\newcommand{\ZZ}{\mathbb{Z}}

\DeclareMathOperator{\Aut}{Aut}

\DeclareMathOperator{\rank}{rank}

\DeclareMathOperator{\Cl}{Cl}
\newcommand{\Jac}{J}
\newcommand{\hhat}{\widehat{h}}
\newcommand{\<}{\langle}
\renewcommand{\>}{\rangle}

\theoremstyle{plain}
\newtheorem{theorem}{Theorem}[section]
\newtheorem{lemma}[theorem]{Lemma}
\newtheorem{proposition}[theorem]{Proposition}
\newtheorem{corollary}[theorem]{Corollary}
\theoremstyle{definition}
\newtheorem{definition}[theorem]{Definition}

\theoremstyle{remark}
\newtheorem{remark}[theorem]{Remark}

\title{Effective height bounds from Mordell--Weil lattice symmetries}
\author{Madhavi Prakash}
\date{}

\begin{document}
\maketitle

\begin{abstract}
We obtain explicit, computable upper bounds for the Néron–Tate height of rational points on curves of genus at least two over number fields, by exploiting the action of automorphisms on the Euclidean lattice structure of the free part of $\Jac(K)$ in the real Mordell–Weil space $V:=\Jac(K)\otimes_{\ZZ}\RR$.  An averaged spectral-gap criterion replaces the 'enough automorphisms' assumption of prior work, and a kernel–injectivity criterion gives a tighter bound whenever some automorphism acts trivially on $V$. When the Jacobian has Mordell–Weil rank 2, a Bravais-lattice test using the height Gram matrix detects this situation. Our approach is demonstrated on a genus-2, rank-2 curve.
\end{abstract}

\section{Introduction}\label{sec:intro}

Mordell's conjecture (proved by Faltings) asserts that a smooth projective curve $X/K$ of genus $g\ge2$ has only finitely many $K$-rational points \cite{Faltings1984}. From a quantitative viewpoint, the natural next step is \emph{effective Mordell}: find an explicit upper bound for the Néron–Tate heights of $K$-points, in terms of invariants of $X$ and $K$. Such height bounds are central because they reduce the search for $X(K)$ to a finite (and, in principle, executable) enumeration inside the Mordell–Weil group of the Jacobian.

In \cite{garciafritz2025effectivemordellcurvesautomorphisms}, Garcia-Fritz and Pasten define a constant \(M(X)\) (see \eqref{eq:def-MX}) and prove an explicit version of Mumford's gap principle (see \ref{thm:gap-principle}), deriving effective height bounds provided the group \(\Aut_K(X)\) is sufficiently large. However, many computationally relevant cases (e.g.\ \((g,\rank)=(2,2)\) with \(|\Aut_K(X)|\le 6\)) fall outside this regime.

This work replaces the large–automorphism requirement with a condition formulated \emph{inside the Mordell–Weil lattice} of the Jacobian, using that automorphisms of the curve act by isometries (for the Néron–Tate pairing) on this lattice. We prove that effective bounds are possible even when the automorphism group is small, provided its action has sufficient “spread’’ on the Mordell–Weil lattice.

The example in §\ref{sec:experiments} illustrates a case where the “enough automorphisms’’ hypothesis fails but our criteria still delivers an effective height bound.

\subsection*{Set-up and notation}
Fix a number field $K$, a curve $X/K$ of genus $g\ge 2$, and its Jacobian $\Jac=\mathrm{Jac}(X)$. Define
\begin{equation}
  j:X\to \Jac,\qquad P\mapsto \Cl\bigl((2g-2)(P)-K_X\bigr). \label{1}
\end{equation}
Here $K_X$ denotes a canonical divisor on $X$.

Let \(\hat{h}\) be the Néron–Tate height on \(J\), normalized to \(K\), associated to \(2\Theta\), where \(\Theta\) is the theta divisor class on \(J\).

Let $V:=\Jac(K)\otimes_{\ZZ}\RR$ be the real Mordell–Weil space equipped with the Néron–Tate pairing (canonical height pairing) \cite[Ch.~9]{HindrySilverman2000} 
$\langle\, ,\,\rangle:V\times V\to\RR$, so that $\|v\|^2=\hhat(v)$.  
\vspace{1em}

Write $G=\Aut_K(X)$. For a subgroup $H\le G$, $\mathrm{Stab}_H(P)$ denotes the stabilizer of $P\in X(\overline{K})$.

The \emph{orthogonal group} of $(V,\langle\, ,\,\rangle)$ is
\[
  O(V) \ :=\ \{F\in\mathrm{End}_{\RR}(V) \mid \langle Fv,Fw\rangle = \langle v,w\rangle\ \text{for all } v,w\in V\},
\]
whose elements are precisely the $\RR$–linear isometries of $V$.

\begin{definition}[Symmetrized operators on $V$]\label{def:symmetrized}
Let $H\le \Aut_K(X)$ be finite. For $\sigma\in H$, let
\[
  \sigma_*:V\to V
\]
denote the $\RR$–linear map induced by pushforward on divisor classes (extended by $\RR$–linearity). Define the self–adjoint operator
\[
  S_\sigma := \tfrac12\bigl(\sigma_*+(\sigma_*)^\dagger\bigr),
\]
and for any probability measure $\mu$ on $H$ with $\mu(\mathrm{id})=0$, set
\[
  S_\mu := \sum_{\sigma\in H}\mu(\sigma)\,S_\sigma,
  \qquad
  \beta_\mu := \lambda_{\min}(S_\mu).
\]
By construction, $S_\sigma$ and $S_\mu$ act on $V$.
\end{definition}

\subsection*{Main results}
Using the above notation, our main result is the following:

\begin{theorem}[Averaged spectral gap]\label{thm:avg-spectral-gap}
If $\beta_\mu>1/g$, then for every $P\in X(K)$ with $\mathrm{Stab}_H(P)=\{\mathrm{id}\}$,
\[
  \hhat\!\left(j(P)\right)\ \le\ \frac{M(X)}{2\bigl(g\,\beta_\mu-1\bigr)}.
\]
\end{theorem}

\noindent\emph{Dirac specialization.}
By applying Theorem~\ref{thm:avg-spectral-gap} with $\mu=\delta_\sigma$ (a Dirac mass at a single automorphism), we obtain:

\begin{corollary}[Dirac spectral–gap criterion]\label{cor:dirac}
Let
\[
  \alpha_H:=\max_{\sigma\in H\setminus\{\mathrm{id}\}}\lambda_{\min}(S_\sigma).
\]
If $\alpha_H>1/g$, then for any $P\in X(K)$ with $\mathrm{Stab}_H(P)=\mathrm{id}$,
\[
  \hhat\bigl(j(P)\bigr)\ \le\ \frac{M(X)}{2\bigl(g\,\alpha_H-1\bigr)}.
\]
\end{corollary}

A particularly favorable situation is when some nontrivial automorphism acts trivially on $V$. In that case $S_\mu=\mathrm{id}$ for a suitable Dirac choice of $\mu$ and the bound improves:

\begin{theorem}[Kernel–injectivity criterion]\label{thm:kernel}
Let $\varphi:\Aut_K(X)\to O(V)$ be $\sigma\mapsto\sigma_*$. If $\ker\varphi\neq\{\mathrm{id}\}$, then for every $P$ with $\mathrm{Stab}_{\ker\varphi}(P)=\{\mathrm{id}\}$,
\[
  \hhat\!\left(j(P)\right)\ \le\ \frac{M(X)}{2g-2}.
\]
\end{theorem}

\noindent\textbf{Detecting the kernel in rank 2.}
When $\rank\,\Jac(K)=2$, the Bravais type of the Mordell–Weil lattice (from the height Gram matrix) bounds the size of its orthogonal symmetry group, giving a concrete test for non-injectivity of $\varphi$:

\begin{corollary}[Rank–2 Bravais test]\label{cor:bravais-rank2}
Assume $\rank \Jac(K)=2$. Let $T$ be the torsion subgroup of $J(K)$ and set $\Lambda=\Jac(K)/T\cong\ZZ^2$ to be the free Mordell–Weil lattice with height Gram matrix $H$ in an LLL-reduced basis. Then
\[
|O(\Lambda)|\in\{2,4,8,12\}
\]
depending on whether $\Lambda$ is oblique, rectangular, square, or hexagonal, respectively. In particular, if $|\Aut_K(X)|>|O(\Lambda)|$, then $\ker\varphi\neq\{\mathrm{id}\}$ and Theorem~\ref{thm:kernel} applies.
\end{corollary}

The proof of Theorem~\ref{thm:avg-spectral-gap} begins from the explicit gap principle, Theorem~\ref{thm:gap-principle} \cite{garciafritz2025effectivemordellcurvesautomorphisms}, which bounds how close two distinct rational points can sit in the Mordell-Weil lattice (the free part of $J(K)$ in $V$). Fix a point and consider all of its images under a finite subgroup of automorphisms; because the stabilizer is trivial, each image is genuinely different, so the gap principle applies to every pair. Averaging these pairwise inequalities with a probability weight on the automorphisms collapses them into a single inequality. To invoke Courant–Fischer, the action is first symmetrized—averaging each automorphism with its adjoint—and then averaged over the subgroup, yielding a self-adjoint operator on the real Mordell–Weil space. A Rayleigh–quotient bound then turns the averaged inequality into a height estimate which is governed by the operator’s smallest eigenvalue.

\section{Background}
\subsection{The explicit gap principle.}
Given a prime \(\frak p\) in \(\mathcal{O}_K\), we define a quantity \(\phi_{\frak p}(X)\) as in \cite{garciafritz2025effectivemordellcurvesautomorphisms}, coming from intersection theory on the fibre at \(\frak p\) of the minimal regular model of \(X\) over \(\mathcal{O}_K\).  This number \(\phi_{\frak p}(X)\) is easily computed from such a model, and it vanishes whenever the special fibre at \(\frak p\) has only one component (in particular, for good reduction).  Furthermore, given an embedding \(v\colon K\hookrightarrow\mathbb{C}\) we let \(X_v\) be the Riemann surface of the \(\mathbb{C}\)-points of \(X\) via \(v\), and we write \(\delta(Y)\) for the Faltings \(\delta\)-invariant of a Riemann surface \(Y\) of positive genus.

We put together these invariants by defining the following quantity:
\begin{equation}\label{eq:def-MX}
\begin{aligned}
M(X)\;=\;&
\frac{(g-1)^2}{3}\,\max\{6,g+1\}
\sum_{v:K\to\mathbb{C}}\delta(X_v)
\;+\;
2\,(g+1)\!\sum_{\frak p}\phi_{\frak p}(X)\,{\log}N\frak p\\
&\quad
+\;
2\,[K:\mathbb{Q}]\,g\,(g-1)^2\bigl(3g\log g +16\bigr).
\end{aligned}
\end{equation}

We recall the form of the gap principle we use from \cite[thm:gap-principle]{garciafritz2025effectivemordellcurvesautomorphisms}:
\begin{theorem}[Explicit gap principle] \label{thm:gap-principle}
    For all pairs of different points \(P,Q\in X(K)\) we have
\[
\hat h\bigl(j(P)\bigr)\;+\;\hat h\bigl(j(Q)\bigr)
\;-\;2\,g\,\langle j(P),\,j(Q)\rangle
\;\ge\;-\,M(X).
\]
\end{theorem}
In particular, if $\hhat(j(P))=\hhat(j(Q))=h$, then 
\begin{equation}\label{eq:gap-rearranged}
  2g\,\< j(P),j(Q)\> - 2h\ \le\ M(X).
\end{equation}

\subsection{Automorphisms as isometries on $V$.}

\begin{definition}[Automorphism action on $V$]\label{def:action}
Let $H\le \Aut_K(X)$ be finite. For $\sigma\in H$, let
\[
  \sigma_*:V\to V
\]
be the $\RR$–linear pushforward induced on divisor classes (extended by $\RR$–linearity).
\end{definition}

\begin{lemma}[Equivariance]\label{lem:eqv}
For every $\sigma\in\Aut_K(X)$ one has
\[
  j\circ\sigma \;=\; \sigma_* \circ j .
\]
\end{lemma}

\begin{proof}
    See \cite[ p.~555, paragraph below (4)]{Fuj97}.
    \[
    \begin{tikzcd}[column sep=huge, row sep=large]
    X(K) \arrow[r,"j"] \arrow[d,"\sigma"'] &
    \Jac(K) \arrow[r,"\otimes_{\ZZ}\RR"] \arrow[d,"\sigma_*"] &
    V \arrow[d,"\sigma_*"] \\
    X(K) \arrow[r,"j"'] &
    \Jac(K) \arrow[r,"\otimes_{\ZZ}\RR"'] &
    V
    \end{tikzcd}
    \]
\end{proof}

\begin{proposition}[Isometry of the action]\label{prop:isometry}
For every $\sigma\in \Aut_K(X)$, the map $\sigma_*$ preserves the Néron–Tate pairing on $V$; in particular $\sigma_*\in O(V)$.
\end{proposition}

\begin{proof}
See \cite[paragraph preceding (5)]{Fuj97}.
\end{proof}

\begin{corollary}[Automorphisms preserve the Néron–Tate height and pairing]
Let $\sigma\in\Aut_K(X)$ and $P,Q\in X(\overline K)$. Then
\[
  \widehat h\!\bigl(j(\sigma P)\bigr)=\widehat h\!\bigl(j(P)\bigr)
  \qquad\text{and}\qquad
  \langle j(\sigma P),\, j(\sigma Q)\rangle=\langle j(P),\, j(Q)\rangle .
\]
\end{corollary}

\begin{proof}
By Lemma~\ref{lem:eqv}, $j(\sigma P)=\sigma_*\,j(P)$ and $j(\sigma Q)=\sigma_*\,j(Q)$. 
By Proposition~\ref{prop:isometry}, $\sigma_*\in O(V)$, so
\[
\langle j(\sigma P),\,j(\sigma Q)\rangle
= \langle \sigma_* j(P),\,\sigma_* j(Q)\rangle
= \langle j(P),\,j(Q)\rangle.
\]
Setting $Q=P$ and using $\widehat h(v)=\langle v,v\rangle$ gives $\widehat h(j(\sigma P))=\widehat h(j(P))$.
\end{proof}

\section{Proof of the main results}\label{sec:spectral-gap}
\noindent

In this section, we prove Theorem~\ref{thm:avg-spectral-gap} and derive its immediate consequences—namely the Dirac specialization (Corollary~\ref{cor:dirac}), the kernel–injectivity criterion (Theorem~\ref{thm:kernel}), and the rank–2 Bravais test (Corollary~\ref{cor:bravais-rank2}). We begin by proving a few preliminary lemmas required for the proof: Lemma~\ref{lem:sym-identity}, the averaged gap inequality (Lemma~\ref{lem:avg-gap}), and the Rayleigh/Courant–Fischer bound (Lemma~\ref{lem:rayleigh}).

We recall Definition~\ref{def:symmetrized} from the Introduction.

\begin{lemma}\label{lem:sym-identity}
For every $\sigma\in H$ and $v\in V$ one has $\langle v,S_\sigma v\rangle=\langle v,\sigma_*v\rangle$.
\end{lemma}
\begin{proof}
By definition,
\[
\langle v,S_\sigma v\rangle
=\tfrac12\bigl(\langle v,\sigma_*v\rangle+\langle v,(\sigma_*)^\dagger v\rangle\bigr)
=\tfrac12\bigl(\langle v,\sigma_*v\rangle+\langle \sigma_*v,\,v\rangle\bigr)
=\langle v,\sigma_*v\rangle,
\]
using the adjoint property and symmetry of $\langle\, ,\,\rangle$.
\end{proof}

\begin{remark}\label{rmk:cosine}
Geometrically, this means on any $\sigma_*$--invariant $2$--plane where $\sigma_*$ is a rotation by angle $\theta$ for the Néron-Tate metric, $S_\sigma$ acts as multiplication by $\cos\theta$. Thus $\langle v,S_\sigma v\rangle/\|v\|^2$ is a cosine, and $\lambda_{\min}(S_\sigma)$ is the smallest cosine over all directions.
\end{remark}

\begin{lemma}[Averaged gap inequality]\label{lem:avg-gap}
Let $H\le \Aut_K(X)$ be finite and $\mu$ a probability measure on $H$ with $\mu(\mathrm{id})=0$. If $P\in X(K)$ has $\mathrm{Stab}_H(P)=\mathrm{id}$ and $v=j(P)$ with $h=\|v\|^2$, then
\begin{equation}\label{eq:avg-gap}
  2g\,\Big\langle v,\sum_{\sigma\in H}\mu(\sigma)\,\sigma_*v\Big\rangle\ -\ 2h\ \le\ M(X).
\end{equation}
\end{lemma}
\begin{proof}
For each $\sigma$ in the support of $\mu$, $\sigma(P)\neq P$, so by \eqref{eq:gap-rearranged}
\(
  2g\,\langle v,\sigma_*v\rangle-2h\le M(X).
\)
Multiply by $\mu(\sigma)$ and sum over $\sigma$; since $\sum\mu(\sigma)=1$, \eqref{eq:avg-gap} follows.
\end{proof}

\begin{lemma}\label{lem:rayleigh}
For any self--adjoint $T$ on $V$ and any $v\in V$,
\[
  \langle v,Tv\rangle\ \ge\ \lambda_{\min}(T)\,\|v\|^2.
\]
\end{lemma}
\begin{proof}
This is the Courant--Fischer variational inequality.
\end{proof}

\begin{remark} Since $\sigma_*\in O(V)$, its spectrum lies on the unit circle and $-1\le \lambda_{\min}(S_\sigma)\le 1$. If $\sigma$ is an involution, then $\sigma_*=(\sigma_*)^{\dagger}$, so $S_\sigma=\sigma_*$. \end{remark}

\vspace{1em}

We are now ready to prove \textbf{Theorem~\ref{thm:avg-spectral-gap}}.
\begin{proof}
Let $v=j(P)$ and $h=\|v\|^2$. By Lemma~\ref{lem:avg-gap},
\[
  2g\,\Big\langle v,\sum_{\sigma}\mu(\sigma)\,\sigma_*v\Big\rangle-2h\ \le\ M(X).
\]
By Lemma~\ref{lem:sym-identity} the LHS equals $2g\,\langle v,S_\mu v\rangle-2h$. Using Lemma~\ref{lem:rayleigh} with $T=S_\mu$ yields
\(
  2g\,\beta_\mu\,h-2h\le M(X),
\)
which rearranges to the stated bound because $\beta_\mu>1/g$.
\end{proof}

\medskip
\noindent\textbf{Why averaging helps.}
The construction of the averaged operator $S_\mu$ is motivated by the fact that averaging cannot worsen the Rayleigh lower bound and can in fact improve it.  
Different $S_\sigma$ may have different “bad directions,” so $\lambda_{\min}(S_\mu)$ can exceed 
$\max_{\sigma}\lambda_{\min}(S_\sigma)$. For example, in $\RR^2$,
\[
T_1=\begin{pmatrix}0&0\\[2pt]0&1\end{pmatrix},\qquad
T_2=\begin{pmatrix}1&0\\[2pt]0&0\end{pmatrix}
\quad\Rightarrow\quad
\lambda_{\min}(T_1)=\lambda_{\min}(T_2)=0
\ \text{ but }\ 
\lambda_{\min}\!\Big(\tfrac12(T_1+T_2)\Big)=\tfrac12.
\]
The next lemma formalizes this.

\begin{lemma}[Minimum eigenvalue under convex combination]\label{lem:convex-min}
If $T_1,\dots,T_m$ are self–adjoint and $\mu_i\ge 0$ with $\sum_i \mu_i=1$, then
\[
  \lambda_{\min}\!\Big(\sum_{i=1}^m \mu_i\,T_i\Big)\ \ge\ \sum_{i=1}^m \mu_i\,\lambda_{\min}(T_i).
\]
\end{lemma}
\begin{proof}
For unit $u$, $\langle u,T_i u\rangle\ge \lambda_{\min}(T_i)$. Hence
\(
\big\langle u,\sum_i\mu_i T_i\,u\big\rangle=\sum_i\mu_i\langle u,T_i u\rangle
\ge\sum_i\mu_i\lambda_{\min}(T_i).
\)
Minimize the left side over unit $u$.
\end{proof}

By applying Theorem~\ref{thm:avg-spectral-gap} with $\mu=\delta_\sigma$ at a maximizer $\sigma$, we get Corollary~\ref{cor:dirac}.

\vspace{2em}
If some nontrivial automorphism acts trivially on $V$, the averaged operator in Theorem~\ref{thm:avg-spectral-gap} becomes the identity and the bound improves to $M(X)/(2g-2)$. Next, we prove Theorem~\ref{thm:kernel} and then explain a practical rank–2 test (Cor.~\ref{cor:bravais-rank2}) using the Bravais classification of the Mordell–Weil lattice to detect non-injectivity of $\varphi$.

\begin{proof}[Proof of Theorem~\ref{thm:kernel}]
Pick $\sigma\in\ker\varphi\setminus\{\mathrm{id}\}$ and set $H:=\langle\sigma\rangle$. 
Let $\mu:=\delta_\sigma$ be the Dirac probability on $\sigma$ (so $\mu(\mathrm{id})=0$).
Since $\sigma\in\ker\varphi$, the induced action on $V$ is $\sigma_*=\mathrm{id}$, hence
\[
S_\mu \;=\; S_\sigma \;=\; \tfrac12\bigl(\sigma_*+(\sigma_*)^\dagger\bigr)\;=\;\mathrm{id},
\qquad
\beta_\mu \;=\; \lambda_{\min}(S_\mu)\;=\;1 \;>\; \frac1g.
\]
Because $\mathrm{Stab}_{\ker\varphi}(P)=\{\mathrm{id}\}$, also $\mathrm{Stab}_{H}(P)=\{\mathrm{id}\}$.
Therefore Theorem~\ref{thm:avg-spectral-gap} applies to $(H,\mu)$ and yields
\[
\hhat\!\bigl(j(P)\bigr)\ \le\ \frac{M(X)}{2\bigl(g\,\beta_\mu-1\bigr)}
\;=\; \frac{M(X)}{2(g-1)}
\;=\; \frac{M(X)}{2g-2}.
\]
\end{proof}

\subsection{Detecting non-injectivity in practice.}

\vspace{0.5em}
Let $\Lambda:=\Jac(K)/T$ be the free part of $\Jac(K)$ and fix free generators $G_1,\dots,G_r$. 
If for some $\sigma\in\Aut_K(X)$ one has $\hhat(\sigma_*(G_i)-G_i)=0$ for all $i$, then each $\sigma_*(G_i)-G_i$ is torsion (the canonical height vanishes only on torsion). 
Hence $\sigma_*$ acts trivially on $\Lambda$, so $\sigma\in\ker\varphi$.

\medskip
\noindent
Complementary to the above check is the following: the natural homomorphism
\[
  \varphi:\Aut_K(X)\longrightarrow O(\Lambda)\subseteq O(V),\qquad \sigma\longmapsto \sigma_*,
\]
has image contained in the point group $O(\Lambda)$ of the Mordell–Weil lattice. 
Therefore
\begin{equation}\label{eq:kernel-lb}
  |\mathrm{im}\,\varphi|\ \le\ |O(\Lambda)| 
  \qquad\Rightarrow\qquad
  |\ker\varphi|\ =\ \frac{|\Aut_K(X)|}{|\mathrm{im}\,\varphi|}\ \ge\ \frac{|\Aut_K(X)|}{|O(\Lambda)|}.
\end{equation}
If $|\Aut_K(X)|>|O(\Lambda)|$, then $\ker\varphi\neq\mathrm{id}$ and Theorem~\ref{thm:kernel} applies.

\vspace{2em}

When $\rank\,\Jac(K)=2$, $|O(\Lambda)|$ is completely explicit: it depends only on the Bravais type of $\Lambda$, which can be read off from the $2\times2$ height Gram matrix in an LLL-reduced basis.

\medskip
In two dimensions there are four crystal systems (oblique, rectangular, square, hexagonal) and \emph{five} Bravais lattices, the extra type being the \emph{centered rectangular} lattice; see the International Tables for Crystallography \cite[Vol.~A, §3.1]{ITA} and the MIT 6.730 notes \cite[Lecture~6, slide~2]{MIT6730L6}. 
For our purposes the two rectangular Bravais types (primitive and centered) share the same point group $D_2$ and hence the same orthogonal symmetry size $|O(\Lambda)|=4$, so we group them together under “rectangular’’ in the list below.

\medskip
Assume $\rank\,\Jac(K)=2$. Let $\Lambda=\Jac(K)/T\cong\ZZ^2$ and choose an LLL-reduced (i.e. minimal-length in rank~2) basis $\{G_1,G_2\}$ of $\Lambda$. 
Write the Gram matrix of the pairing as
\[
  H\ =\ \begin{pmatrix}H_{11}&H_{12}\\ H_{12}&H_{22}\end{pmatrix}
  \quad\text{with}\quad H_{ij}=\<G_i,G_j\>.
\]
The lattice type is determined by $(H_{11},H_{22},H_{12})$:
\begin{itemize}[leftmargin=*]
  \item \textbf{Oblique:} $H_{12}\ne0$, $H_{11}\ne H_{22}$, and $\dfrac{H_{12}}{\sqrt{H_{11}H_{22}}}\ne\pm\tfrac12$. Then $O(\Lambda)\cong C_2$ and $|O(\Lambda)|=2$.
  \item \textbf{Rectangular:} $H_{12}=0$ and $H_{11}\ne H_{22}$. Then $O(\Lambda)\cong D_2$ and $|O(\Lambda)|=4$.
  \item \textbf{Square:} $H_{12}=0$ and $H_{11}=H_{22}$. Then $O(\Lambda)\cong D_4$ and $|O(\Lambda)|=8$.
  \item \textbf{Hexagonal:} $H_{11}=H_{22}$ and $H_{12}=\tfrac12 H_{11}$. Then $O(\Lambda)\cong D_6$ and $|O(\Lambda)|=12$.
\end{itemize}

\noindent
Combining the table above with \eqref{eq:kernel-lb} yields the stated rank–2 test: if $|\Aut_K(X)|$ exceeds the corresponding $|O(\Lambda)|$ for the Bravais type read from $H$, then $\ker\varphi\neq\mathrm{id}$ and Theorem~\ref{thm:kernel} gives the stronger height bound.

\section{Example}\label{sec:experiments}

Consider the genus-$2$ curve over $\QQ$ with LMFDB label \texttt{196098.a.196098.1}
\[
  X:\ y^2 = x^6 - 12x^4 + 6x^3 - 284x^2 + 1488x - 1815
\]
\[
\Aut_{\QQ}(X)\cong C_2^2
\]

\textbf{The height bound of} \textbf{\cite{garciafritz2025effectivemordellcurvesautomorphisms} }coming from the “enough automorphisms’’ hypothesis (see the table on p.~2 of their paper) \textbf{does not apply here} because $|\Aut_\QQ(X)|=4$ is strictly smaller than the required threshold in the $(g,r)=(2,2)$ regime.

\vspace{2em}
The Mordell--Weil group of the Jacobian
\[
  \Jac(\QQ)\ \cong\ \ZZ\oplus\ZZ\oplus (\ZZ/4\ZZ).
\]
Using \textsc{Magma} \cite{MR1484478}, free generators $G_2,G_3$ for the rank-$2$ part can be chosen and the pushforwards along the automorphism
\[
  [X:Y:Z]\ \longmapsto\ \Bigl[X-\tfrac{21}{8}Z:\ \tfrac{1}{512}Y:\ \tfrac{3}{8}X - Z\Bigr].
\]
computed. The data are:
\begin{align*}
G_2\ &(\text{Generator 2}) &&=\ (1,\ x^3,\ 2),\\
F_2\ &(\text{Pushforward of }G_2) &&=\ \Bigl(x^2-\tfrac{16}{3}x+\tfrac{64}{9},\ \tfrac{8}{3}x-\tfrac{191}{27},\ 2\Bigr),\\
\hhat(G_2)\ &=\ 2.117,\\
\hhat(F_2)\ &=\ 2.117,\\
\hhat(G_2 - F_2)\ &=\ 0,\\[0.5em]
G_3\ &(\text{Generator 3}) &&=\ \Bigl(x^2+6x-23,\ \tfrac{99}{2}x-\tfrac{263}{2},\ 2\Bigr),\\
F_3\ &(\text{Pushforward of }G_3) &&=\ \Bigl(x^2+6x-23,\ \tfrac{99}{2}x-\tfrac{263}{2},\ 2\Bigr),\\
\hhat(G_3)\ &=\ 3.324,\\
\hhat(F_3)\ &=\ 3.324,\\
\hhat(G_3 - F_3)\ &=\ 0.
\end{align*}
Hence this automorphism acts trivially on $V$, so $\ker\varphi\ne\mathrm{id}$.

Another way to check if $\ker\varphi\ne\mathrm{id}$ is via the height Gram matrix in the reduced basis $\{G_2,G_3\}$, 
\[
  H\ =\ \begin{pmatrix}
     2.116 & -0.913\\
    -0.913 & 3.324
  \end{pmatrix},
\]
so $H_{12}\ne0$, $H_{11}\ne H_{22}$, and $\dfrac{H_{12}}{\sqrt{H_{11}H_{22}}}\approx -0.344\ne\pm\tfrac12$. Thus $\Lambda$ is oblique and $|O(\Lambda)|=2$, while $|\Aut(X)|=4$, consistent with~\eqref{eq:kernel-lb}.

Since $\ker\varphi\ne\mathrm{id}$, Theorem~\ref{thm:kernel} applies to this genus-2 curve. This gives us the height bound
$\widehat h(j(P))\le M(X)/2$, thus enabling us to find all the rational points on $X$.

\section*{Acknowledgements}
\noindent This project grew out as a part of the work done during the UC Berkeley Mathematics REU 2025. 
The author is grateful to Yinglu Zhou and Eric Ju for valuable conversations and Fangu Chen for guidance and mentorship. 
This work was supported by the National Science Foundation (DMS–2342225) and the Mathematical and Physical Sciences Foundation.

\thispagestyle{empty}
\bibliographystyle{amsalpha}
\bibliography{main}

\providecommand{\bysame}{\leavevmode\hbox to3em{\hrulefill}\thinspace}
\providecommand{\MR}{\relax\ifhmode\unskip\space\fi MR }
\providecommand{\MRhref}[2]{%
  \href{http://www.ams.org/mathscinet-getitem?mr=#1}{#2}
}
\providecommand{\href}[2]{#2}
\begin{thebibliography}{GFP25}

\bibitem[BCP97]{MR1484478}
Wieb Bosma, John Cannon, and Catherine Playoust, \emph{The {M}agma algebra system. {I}. {T}he user language}, J. Symbolic Comput. \textbf{24} (1997), no.~3-4, 235--265, Computational algebra and number theory (London, 1993). \MR{MR1484478}

\bibitem[Fal84]{Faltings1984}
Gerd Faltings, \emph{Errata: Endlichkeitss{\"a}tze f{\"u}r abelsche variet{\"a}ten {\"u}ber zahlk{\"o}rpern}, Inventiones mathematicae \textbf{75} (1984), 381.

\bibitem[Fuj97]{Fuj97}
Masanari Fujimori, \emph{Rational points of a curve which has a nontrivial automorphism}, Annali della Scuola Normale Superiore di Pisa - Classe di Scienze \textbf{24} (1997), no.~3, 551--569.

\bibitem[GFP25]{garciafritz2025effectivemordellcurvesautomorphisms}
Natalia Garcia-Fritz and Hector Pasten, \emph{Effective mordell for curves with enough automorphisms}, 2025.

\bibitem[HA16]{ITA}
Th. Hahn and M.~I. Aroyo (eds.), \emph{International tables for crystallography, volume a: Space-group symmetry}, 6 ed., International Union of Crystallography / Wiley, 2016, See Sec.~3.1: two-dimensional Bravais lattices (five types).

\bibitem[HS00]{HindrySilverman2000}
Marc Hindry and Joseph~H. Silverman, \emph{Diophantine geometry: An introduction}, Graduate Texts in Mathematics, vol. 201, Springer, New York, 2000.

\bibitem[oP04]{MIT6730L6}
MIT~Department of~Physics, \emph{6.730 physics for solid state applications, lecture 6: Periodic structures}, \url{https://web.mit.edu/6.730/www/ST04/Lectures/Lecture6.pdf}, 2004, Slide “2D Bravais Lattices”: square, rectangular, hexagonal, oblique, centered rectangular.

\end{thebibliography}

\end{document}